\documentclass[reqno,10pt]{amsart}
\usepackage{amsmath,amsthm,amssymb,amsfonts,color,,xcolor,pifont,cite}
\usepackage{a4wide}
\usepackage[colorlinks]{hyperref}
\usepackage[showonlyrefs]{mathtools}
\mathtoolsset{showonlyrefs=true}
\usepackage{framed,enumerate}
\usepackage[normalem]{ulem}

\newcounter{corr}
\definecolor{violet}{rgb}{0.580,0.,0.827}
\newcommand{\corr}[3]{\typeout{Warning : a correction remains in page
		\thepage}
	\stepcounter{corr}        
	{\color{blue}\ifmmode\text{\,\sout{\ensuremath{#1}}\,}\else\sout{#1}\fi}
	{\color{red}#2}
	{\color{violet} #3}}

\usepackage{pgf,tikz}
\usepackage{mathrsfs,ulem}
\usetikzlibrary{arrows}

\newcommand{\cm}[1]{{\color{blue}{#1}}}

\newtheorem{theorem}{Theorem}[section]
\newtheorem{proposition}[theorem]{Proposition}
\newtheorem{corollary}[theorem]{Corollary}
\newtheorem{lemma}[theorem]{Lemma}

\newtheorem{remark}[theorem]{Remark}
\newtheorem{definition}[theorem]{Definition}

\newcommand{\R}{\mathbb R}

\newcommand{\bS}{\mathbb{S}}

\def\I{\mathcal I}

\def\H{\mathcal H}
\def\eps{\varepsilon}
\def\lt{\left}
\def\rt{\right}
\def\S{\mathbb{S}}
\def\les{\lesssim}
\def\ges{\gtrsim}
\def\Ia{\I_\alpha}
\def\Iae{\I_{\alpha}^{\eta}}

\def\F{\mathcal F}

\def\wF{\widetilde{\F}}

\def\g{\mathbf{g}}
\def\M{\mathcal{M}}
\def\Msc{\M_{\rm sc}}
\def\spt{{\rm Spt}}
\def\LM#1{\hbox{\vrule width.2pt \vbox to#1pt{\vfill \hrule width#1pt
height.2pt}}}
\def\LL{{\mathchoice {\>\LM7\>}{\>\LM7\>}{\,\LM5\,}{\,\LM{3.35}\,}}}
\def\restr{{\LL}}
\def\bSr{\bS_r^2}

\author{Michael Goldman}
\address{CMAP, CNRS, \'Ecole polytechnique, Institut Polytechnique de Paris, 91120 Palaiseau,
	France}
\email{michael.goldman@cnrs.fr}

\author{Matteo Novaga}
\address{Department of Mathematics, University of Pisa, 56127 Pisa, Italy} 
\email{matteo.novaga@unipi.it} 

\author{Berardo Ruffini}
\address{Department of Mathematics, University of Bologna, 40126 Bologna, Italy} 
\email{berardo.ruffini@unibo.it}

\numberwithin{equation}{section}

\title[A liquid drop model with Willmore energy]{A charged liquid drop model with Willmore energy}

\begin{document}
	
	\begin{abstract}
		We consider a variational model of electrified liquid drops, involving competition between surface tension and charge repulsion. 
		Since the natural model happens to be ill-posed, we show that by adding to the perimeter a Willmore-type energy,  the problem turns back to be well-posed. We also prove that for small charge the droplets is spherical.
	\end{abstract}
	\maketitle
	\tableofcontents

	\section{Introduction}

	The model proposed in 1882 by Lord Rayleigh \cite{Ray} is widely used to describe the equilibrium shapes of charged liquid droplets in presence of an electrical charge, see e.g. \cite{miksis,roth,delamora,FonFri}. The underlining energy is given by
	\[
	 P(E)+\frac{Q^2}{{\rm Cap}_2(E)}.
	\] 
	Here  $E$ is a subset of $\R^3$, $P$ is a perimeter term modeling surface tension, $Q>0$ is the amount of charge and ${\rm Cap}_2$ is the standard capacity. The observed droplets should be (at least for small $Q$) stable volume-constrained critical points of this energy. Quite surprisingly we showed in \cite{gnrI} that this model is ill-posed in the sense that no local minimizers for the natural $L^1$ topology exist at any charge.  To reconcile this with the experimental and numerical observation of stable charged liquid drops, substential effort has been made to understand possible regularizing mechanisms. One possibility is to  restrict a priori the class of candidates as in \cite{gnrI,gnrII}. Another possibility is to add regularizing terms to the energy. See also \cite{muratov2023variational} where both aspects are present. One way of regularizing the energy  is to penalize strong concentration of charges (which amounts to regularize the capacity term) as in  \cite{MurNov,DHV,muratov2023variational} and to some extent \cite{MurNovRuf,MurNovRuf2,gnr4}. A second very natural way to regularize the energy is instead to take into account higher order effects in the term modelling  surface tension. It is for instance suggested in \cite{MichaelGoldman2017} that penalizing the curvature through the Willmore energy could be enough to restore well-posedness. The aim of this paper is to answer this question affirmatively.\\

\subsection{Notation and main results}
We will use the notation $A\lesssim B$ to indicate that there exists a constant $C>0$, typically 
depending on the dimension $d$ and on $\alpha\in (0,d)$ such that $A\le C B$ (we will specify when $C$ depends on other quantities).  We  write $A\sim B$ if $A\lesssim B\lesssim A$, and $A\ll B$ to indicate that there exists a (typically small) universal constant $\eps>0$ such that $A\le \eps B$.\\
 
	
	For $d=2,3$, $\alpha\in(0,d)$, $Q\ge 0$, $\lambda\ge0$  and $E$ a smooth set we consider the energy
	\[
	 \F_{\lambda,Q}(E)=\lambda P(E)+W(E) +Q^2 \I_\alpha(E).
	\]
When $\lambda=0$ we simply write $\F_Q(E)=\F_{0,Q}(E)=W(E) +Q^2 \I_\alpha(E)$. The various terms of the energy are:
\begin{itemize}
 \item the perimeter $P$ (see \cite{Maggi}), which satisfies for smooth sets
\begin{equation*}
	P(E) = \H^{d-1}(\partial E), 
\end{equation*}
where $\mathcal H^{d-1}$ denotes the $(d-1)$-dimensional Hausdorff measure in $\R^d$;
\item the elastic or Willmore energy $W$, defined as
\begin{equation*}	W(E) = 
	\begin{cases}
		\displaystyle\int_{\partial E} H^2\,d\H^1 \quad&\text{ for }d=2,
		\\
		\\
		\displaystyle\frac{1}{4}\int_{\partial E} H^2\,d\H^2\quad&\text{ for }d=3,
	\end{cases}
\end{equation*}
where  $H$ denotes the mean curvature of $\partial E$, i.e.~the curvature in dimension two and the sum of the principal curvatures in dimension three (see Section \ref{sec:prel});\\
\item the Riesz interaction energy $\I_\alpha$ defined, for $\alpha\in(0,d)$, as
\begin{equation}\label{defIaE}
	\I_\alpha(E) =\min_{\mu(E)=1} I_\alpha(\mu)=\frac{1}{{\rm Cap}_\alpha(E)}
\end{equation}
with
\begin{equation}\label{defIamu}
 I_\alpha(\mu)=\int_{\R^d\times \R^d} \frac{d\mu(x)d\mu(y)}{|x-y|^{d-\alpha}}.
\end{equation}
\end{itemize}
Given a volume constant $m>0$, we consider then the following problem
	\begin{equation}\label{eq:main}
		\min\left\{ \F_{\lambda,Q}(E)\,:\, |E|=m \right\}.
	\end{equation}
A first step in the study of \eqref{eq:main} was taken in \cite{GolNovRog}. There, the capacitary term was replaced by
\[
 V_\alpha(E)=\int_{E\times E}\frac{dxdy}{|x-y|^{d-\alpha}}.
\]
This corresponds to the assumption that the charge distribution $\mu$ is uniform on $E$. The functional is then a perturbation of the Gamow type models studied for instance in \cite{KM2,f2m3,frank2021existence}. The main results of \cite{GolNovRog} may be summarized as follows. First if we consider the case $d=2$, disks are the unique minimizers among simply connected sets for $Q$ small enough and $\lambda=0$ (and thus also for any $\lambda\ge 0$ by the isoperimetric inequality). Dropping the assumption that the sets are simply connected we lose existence for $\lambda=0$. For $\lambda>0$ there is now a dichotomy. There exists $\bar \lambda=\bar{\lambda}(m)>0$ such that for $Q=0$, minimizers are disks if $\lambda\ge \bar \lambda$ while they are annuli if $\lambda\le \bar \lambda$. This picture remains valid if $Q>0$ is small enough. When $d=3$ and $\lambda=0$ (and thus again for every $\lambda\ge 0$ by isoperimetric inequality), it was proven that for small enough $Q$ and $\alpha\in (1,3)$ the corresponding energy is uniquely minimized by balls. In this paper we prove that to a large extent the same results hold for \eqref{eq:main}.
	
In the two-dimensional case we strongly rely on the analysis from \cite{GolNovRog}. Let us define
	\begin{equation}\label{prob2dsc}
		 \min_{E\in \Msc(m)} \F_{\lambda,Q}(E),
		\end{equation}

	\begin{equation}\label{prob2d}
		 \min_{E\in \M(m)} \F_{\lambda,Q}(E),
		\end{equation}
		where $\M(m)$ is the family of measurable sets of measure $m$ in $\R^2$ and $\Msc(m)$ is the subset of $\M(m)$ of simply connected sets.
	\begin{proposition}\label{prop:Msc}
	 For every $\alpha\in(0,2)$, there exists $Q_0=Q_0(\alpha)$ such that for every $\lambda\ge 0$ and every $Q\le Q_0 m^{-(\alpha-1)/2}$, the only minimizers of \eqref{prob2dsc} are balls.
	\end{proposition}
	For \eqref{prob2d} we are only able here to treat the case $\lambda>\bar{\lambda}$ when we expect minimizers to be disks when they exist.  We set here $\bar \lambda=\bar \lambda(|B_1|)$ to be the constant  given by \cite[Theorem 2.7]{GolNovRog}.
\begin{proposition}\label{prop:M}
  For every $\alpha\in (0,2)$ and $\lambda_0>\bar \lambda$, there exists $Q_0=Q_0(\lambda_0,\alpha)>0$ such that for every $\lambda\ge \lambda_0 m^{-1}$ and every $Q\le Q_0 m^{-(\alpha-1)/2}$, the only minimizers of \eqref{prob2d} are balls.
\end{proposition}
\begin{remark}
 When $\lambda\le \bar \lambda$ we expect to see instead annuli. However this seems to be a difficult problem as the analysis in \cite{GolNovRog} heavily relies on the fact that $V_\alpha(E)$ is an explicit function of $E$.
\end{remark}
	We now turn to the case $d=3$. Our  main result is the following.
	\begin{theorem}\label{thm:main}
		For every $\alpha\in[2,3)$, there exists a constant $\overline Q=\overline{Q}(\alpha)>0$ such that for every $\lambda\ge 0$ and $m>0$ with $Q\le \overline{Q}  m^{(3-\alpha)/6}$, the only minimizers of \eqref{eq:main} are the balls of measure $m$.
	\end{theorem}
	 As above we prove the result for $\lambda=0$. Moreover, by scaling we may assume that $m=|B_1|$. The main difference between $d=2$ and $d=3$ is that in the former, having finite energy implies a $C^{1,1/2}$ control on the sets implying in particular that sets of small energy must be nearly spherical. In the latter, we can still obtain from \cite{DeLMu2005} that they are $W^{2,2}\cap {\rm Lip}$ parameterizations of the sphere. While a stability inequality for $W$ is given by \cite{roesc12}, as opposed to \cite{GolNovRog}, this does not seem to be  enough to perform the Fuglede type computations from \cite{gnrI,Prunier,gnr4} for $\Ia$. Our strategy is thus to first prove (uniform in $Q$) $C^{1,\beta}$ regularity for minimizers. A first issue is the existence of minimizers. If we work with the non-parametric approach of Simon, see \cite{Simon93}, we naturally end up in the class of varifolds with square integrable mean curvature. As in \cite{gnrI} a major problem is then the semi-continuity of $\Ia$. We solve this issue by relying on \cite{DeLMu2005,DeLMu2006} to obtain that in the small charge regime, sets of small energy must be parameterizations of the sphere with controlled $W^{2,2}$ norm. This yields convergence in the Hausdorff sense from which we obtain existence of minimizers. In order to prove regularity of these minimizers, a natural idea would be to follow the parametric approach \cite{riviere2014variational} of Rivi\`ere. However, this does not seem to be easily compatible with our capacitary term. We go back instead to the non-parametric approach viewing the volume penalization and $\Ia$ as lower order perturbations of $W$. Our proof departs in a substantial point from \cite{Simon93} (see also \cite{pozzetta2017confined}). Indeed, a central difficulty in the theory of varifolds is that it is in general not known that smooth sets are dense in energy. Since the construction of competitors in \cite{Simon93} has to be made for smooth sets, \cite{Simon93} does it at the level of minimizing sequences. The main drawback is that one has to distinguish the 'good points' where there is no energy concentration from the 'bad points' where there is energy concentration. In our case we leverage on the fact that sets with small energy are parametrized to prove density in energy of smooth sets (see Lemma \ref{lem:approx}). 
Using this approximation result we can avoid altogether the presence of the 'bad points'. Not only this simplifies significantly the proof from \cite{Simon93},
but it also allows to obtain bounds which do not depend on $Q$. Indeed, the analysis of the bad points is based on a compactness argument which does not seem to yield uniform regularity estimates. We believe that one of the original aspects of this work is the combined use of the parametric and non-parametric approaches together.

 The restriction $\alpha\in [2,3)$ in Theorem \ref{thm:main} comes from the fact that we use the Fuglede type estimates from \cite{gnrI,Prunier,gnr4} which are not known for $\alpha<2$. Indeed, in this case the underlying operator is a fractional Laplacian of order larger than one. Not much seems to be known for this type of harmonic measures.

\begin{remark}
 Our proof of existence and  $C^{1,\beta}$ regularity for minimizers of \eqref{eq:main} would also apply (and would actually be easier) for the functional from \cite{GolNovRog} (replacing $\Ia$ by $V_\alpha$). In particular, this proves that for small $Q$, the corresponding minimizers are nearly-spherical. Using the Fuglede computations from \cite{KM2,f2m3} we can extend \cite[Theorem 4.6]{GolNovRog} to the whole range $\alpha\in(0,3)$.
\end{remark}

\smallskip

\noindent{\bf Acknowledgments.}
	The authors thank M. Pozzetta for useful discussions about Lemma \ref{lem:approx}. We thank F. Rupp for pointing out \cite{Rupp24} to us.
	M. Goldman and B. Ruffini  were partially supported by the project ANR-18-CE40-0013 SHAPO, financed by the French Agence Nationale de la Recherche (ANR), and M. Novaga was partially supported by the project PRIN 2022 GEPSO, financed by the European Union -- NextGenerationEU.
	B. Ruffini and M. Novaga are members of the INDAM-GNAMPA. M.Goldman  and M. Novaga acknowledge the project G24-202 ``Variational methods for geometric and optimal matching problems'' funded by Università Italo Francese.

	\section{The two-dimensional case}
We first prove Proposition \ref{prop:Msc}.
	\begin{proof}[Proof of Proposition \ref{prop:Msc}]
	 By the isoperimetric inequality it is enough to prove the statement for $\lambda=0$. By scaling we may further assume that $m=|B_1|$. If $E\in \Msc(|B_1|)$ is such that $\F_{0,Q}(E)\le \F_{0,Q}(B_1)$, then
	 \[
	  W(E)\le W(B_1) +Q^2(\I_\alpha(B_1)-\I_\alpha(E)).
	 \]
In particular, if $Q$ is small enough, we can apply \cite[Theorem 2.3]{GolNovRog} to obtain
\begin{equation}\label{quantWI}
 Q^2(\I_\alpha(B_1)-\I_\alpha(E))\ge W(E)-W(B_1)\ges (P(E)-P(B_1)) +\min_{x\in \R^2} |E\Delta B_1(x)|.
\end{equation}
Since $\partial E=\gamma ([0,P(E)])$ where $\gamma$ is a unit-speed parametrization with
\[
 W(E)=\int_0^{P(E)} |\gamma''|^2,
\]
from the embedding of $W^{2,2}([0,P(E)])$ in $C^{1,1/2}([0,P(E)])$ we see that for every $\beta<1/2$ and every $Q$ small enough, up to translation, every such set is $C^{1,\beta}$ close to $B_1$ and thus a nearly spherical set. By \cite[Proposition 4.5]{gnr4}  we then get
\[
 \I_\alpha(B_1)-\I_\alpha(E)\les P(E)-P(B_1).
\]
Combining this and \eqref{quantWI} yields a contradiction if $Q$ is small enough.
	\end{proof}
We now prove Proposition \ref{prop:M}.
\begin{proof}[Proof of Proposition \ref{prop:M}]
 By the isoperimetric inequality it is enough to prove that for every $\lambda>\bar \lambda$ there exists $Q_0>0$ such that balls are the unique volume-constrained minimizers of $\F_{\lambda,Q}$ for every $Q\le Q_0$. For this we notice as before that if $\F_{\lambda,Q}(E)\le \F_{\lambda,Q}(B_1)$, then
 \[
  \F_{\lambda,0}(E)\le \F_{\lambda,0}(B_1) +Q^2 (\I_\alpha(B_1)-\I_\alpha(E))
 \]
we may thus apply \cite[Lemma 2.9]{GolNovRog} and conclude that for $Q$ small enough (depending on $\lambda$), such sets must be simply connected. By Proposition \ref{prop:Msc}, this concludes the proof.
\end{proof}

	\section{The three-dimensional case}
	\subsection{Preliminaries on Riemannian geometry and weak immersions}\label{sec:prel}
	
In this section we first collect some notation and well-known facts from Riemannian geometry, see e.g. \cite{docarmosurf,do1992riemannian}. We then prove in Lemma \ref{lem:approx} an approximation result for weak immersions. For $r>0$ we write $\bSr=\partial B_r$ and simply $\S^2$ for $\S^2_1$. \\
We denote by $D$ the flat connection on $\R^3$ i.e. if $X=\sum_{i=1}^3 X_i e_i$ then $D_X Y=\sum_{i=1}^3 X_i \partial_i Y$. Let $\Sigma\subset \R^3$ be a smooth compact $2-$dimensional manifold oriented by its normal $\nu$. For $X,Y$ tangent vectorfields to $\Sigma$ we define the action of the second fundamental form $A$ on them as
\begin{equation}\label{defA}
 A(x)[X(x),Y(x)]=- (D_X Y)\cdot \nu.
\end{equation}
Notice that witht this convention, if $\Sigma$ is the boundary of a convex set oriented with the outward normal then $A$ is positive. If $(\tau_1,\tau_2)$ is a local orthogonal frame we have for $i=1,2$,
\begin{equation}\label{Dtaui}
 D_{\tau_i} \tau_j=- A[\tau_i,\tau_j] \nu.
\end{equation}
Moreover, if $\Sigma=\S^2$, we have $A[\tau_i,\tau_j]=\delta_{ij}$. When considering $A$ as a (symmetric) linear map from $T_x\Sigma$ to $T_x\Sigma$ we write
\[
 H={\rm Tr}(A) \qquad \textrm{and } \qquad A^{\circ}=A -\frac{1}{2}({\rm Tr } A)\,  {\rm Id}
\]
to be the mean curvature and the  traceless part of $A$. If $(\kappa_1,\kappa_2)$ are the eigenvalues of $A$, we have
\[
 |H|^2= (\kappa_1+\kappa_2)^2, \qquad |A|^2=\kappa_1^2+\kappa_2^2 \qquad \textrm{and } \qquad |A^\circ|^2=\frac{1}{2} (\kappa_1-\kappa_2)^2.
\]
If we denote by $\g(\Sigma)$ the genus of $\Sigma$, we have by the Gauss-Bonnet formula,
\begin{multline}\label{GaussBonnet}
 \frac{1}{4}\int_{\Sigma} |H|^2 d\H^2=\frac{1}{4}\int_{\Sigma} |A|^2d\H^2 + 2\pi (1-\g(\Sigma)) \qquad \textrm{and } \\ \int_{\Sigma} |A^\circ|^2d\H^2 =\frac{1}{2}\int_{\Sigma} |A|^2 d\H^2-4\pi (1-\g(\Sigma)).
\end{multline}
In particular we have as consequence of the Li-Yau inequality (and $\g(\Sigma)\ge 0$) that for any smooth  compact surface $\Sigma$,
\begin{equation}\label{LiYauA}
 \frac{1}{4}\int_{\Sigma} |A|^2\ge 2\pi (1+\g(\Sigma))\ge 2\pi.
\end{equation}
 For $\psi$ a smooth real-valued function defined on a neighborhood of $\Sigma$ we define $\nabla \psi (x)\in T_x\Sigma$ as the line vector
\[
 \nabla \psi= \pi_{T_x\Sigma} (D\psi)
\]
where $\pi_{T_x\Sigma}$ is the projection on $T_x\Sigma$. If $(\tau_1,\tau_2)$ is a local orthonormal frame, we often write $\partial_i \psi= \nabla\psi\cdot \tau_i$. We then define
\begin{equation}\label{defnabla2}
 \nabla^2 \psi(x)[X(x),Y(x)]=D^2 \psi(x)[X(x),Y(x)]-A[X,Y](x)D\psi(x)\cdot \nu(x).
\end{equation}
We similarly denote $\partial_{ij} \psi= \nabla^2 \psi[\tau_i,\tau_j]$. If $\psi=(\psi_1,\cdots \psi_k)\in \R^k$ then we still write $\nabla \psi$ for the matrix whose $k-th$ line is given by $\nabla \psi_k$. In local coordinates we have $\nabla \psi=(\partial_1 \psi, \partial_2 \psi)$. For a $(k\times 3)$-array $M=(M_1,\cdots,M_k)^T$ we write $\pi_{T_x \Sigma}(M)$ for the $(k\times 3)$-array $(\pi_{T_x \Sigma}(M_1),\cdots \pi_{T_x\Sigma}(M_k))^T$. Let us point out that if $\Sigma$ is connected and $\psi$ is a smooth embedding then by the theorem of Jordan-Brouwer, see \cite[Proposition 12.2]{Benedet}, there exists $E$ such that $\partial E=\psi(\Sigma)$.

	 For  $\psi\in W^{2,2}(\bSr)\cap {\rm Lip}(\bSr)$ such that for some $c_0>0$ and $\H^2$ a.e. on $\bSr$,
	\[
	 |\partial_1 \psi \wedge \partial_2 \psi|\ge c_0
	\]
	 we set
	\[
	 N_\psi=\frac{\partial_1 \psi \wedge \partial_2 \psi}{|\partial_1 \psi \wedge \partial_2 \psi|}.
	\]
	We consider $N_\psi$ as a function from $\bSr$ to $\bS^2$.
	Notice that such $\psi$ are weak immersions in the sense of \cite{riviere2014variational}.
For $\Sigma=\psi(\bSr)$,  we see that $N_\psi$ is a normal vectorfield to $\Sigma$. Let then $g$ be the pullback metric on $\bSr$ of the standard Euclidean metric on $\Sigma$ through $\psi$ i.e.
\[
 g(v,w)= \nabla \psi(v)\cdot \nabla \psi(w) \qquad \forall (v,w)\in T\bSr\times T\bSr.
\]
We then write $d{\rm vol}_g$ for the volume form induced by the metric $g$, that is
\[
 d{\rm vol}_g=\sqrt{|\partial_1 \psi|^2|\partial_2 \psi|^2-(\partial_1\psi\cdot \partial_2 \psi)^2} d\H^{2}=|\partial_1 \psi\wedge \partial_2\psi|d\H^{2}.
\]
We then have (see \cite[Section 3.3]{docarmosurf})
\[
 \int_{\Sigma} |A|^2 d\H^2=\int_{\bSr} | A G^{-1}|^2 \, d {\rm vol}_g
\]
where by \eqref{defA},
\[
 A_{ij}= -(\partial_{ij} \psi)\cdot N_\psi =-(\nabla^2 \psi \cdot N_\psi)_{ij} \]
 and
 \[ G^{-1}=\frac{1}{|\partial_1 \psi\wedge \partial_2\psi|^2}\begin{pmatrix} |\partial_2 \psi|^2 & -\partial_1\psi\cdot \partial_2\psi\\
-\partial_1\psi\cdot \partial_2\psi & |\partial_1 \psi|^2
											\end{pmatrix}.
\]
Notice that by the Weingarten equations, see \cite[Section 3.3]{docarmosurf},
\[
 | A G^{-1}|^2=|\nabla N_\psi|^2_g.
\]
Therefore,
\begin{multline}\label{eq:willmoreinpsi}
 \int_{\Sigma} |A|^2 d\H^2 =\int_{\bSr} |\nabla N_\psi|^2_g \, d {\rm vol}_g\\
 = \int_{\bSr}|\partial_1 \psi\wedge \partial_2\psi|^{-3}\lt|(\nabla^2 \psi \cdot N_\psi)\begin{pmatrix} |\partial_2 \psi|^2 & -\partial_1\psi\cdot \partial_2\psi\\
-\partial_1\psi\cdot \partial_2\psi & |\partial_1 \psi|^2
\end{pmatrix}\rt|^2 d\H^{2}.
\end{multline}
When $\psi\in W^{2,2}(\bSr)$ is conformal, we let $h^2= |\nabla \psi|^2/2=|\partial_1 \psi \wedge \partial_2 \psi|$ be the conformal factor. Notice  that for $\psi$ conformal  if $h^2=|\nabla \psi|^2/2$ is bounded then $\psi\in {\rm Lip}(\bSr)$. For such conformal maps we have
\[
 \int_{\Sigma} |A|^2 d\H^2=\int_{\bSr} |\nabla N_\psi|^2 \, d \H^{2}.
\]

We will need an approximation result. As suggested in \cite[page 316]{KuwertLi2012} we follow the strategy of \cite[Section 4]{SchUhl}. 

 \begin{lemma}\label{lem:approx0}
 	Assume $\psi\in W^{2,2}(\S^2,\R^3)\cap {\rm Lip}(\S^2,\R^3)$ is conformal and  such that
 	\[
 	\|h-1\|_{L^\infty(\S^2)}\le\frac{1}{4}.
 	\]

 	 Then there exists a sequence of smooth maps $\psi_n$ such that $\psi_n$ converges strongly in $W^{2,2}(\S^2,\R^3)$ to $\psi$ with
 	\begin{equation}\label{quasiconfpsin}
 	|\partial_1\psi_n\wedge\partial_2\psi_n|\ge \frac18 \qquad \textrm{and } \qquad \|\nabla\psi_n\|_{L^\infty(\S^2)}\le 2.
 	\end{equation}
  \end{lemma}
 \begin{proof}
	Let us  define the subset of $\R^3\times\R^3$
 	\[
 	M=\left\{ h(v_1,v_2)\,:\, (h,v_1,v_2)\in\mathcal A\right\},
 	\]
 	where $\mathcal A$ is the family of triples $(h,v_1,v_2)\in \R\times\S^2\times \S^2$ such that $|h-1|\le 1/2$, $\langle v_1,v_2\rangle=0$. Notice that for a.e. $x\in\S^2$, $\nabla\psi(x)\in M$ as it is a conformal map. We extend $\psi$ to $\R^3$ by $0-$homogeneity. Notice that if $\pi_{\S^2}(x)=x/|x|$ is the projection on $\S^2$, this means that we identify $\psi$ and $\psi(\pi_{\S^2}(x))$. For further use let us point out that if $\zeta$ is $0-$homogeneous then $D\zeta(x)= \frac{1}{|x|} \nabla \zeta (\pi_{\S^2}(x))$. Writing in coordinates that
 	\[
 	 \nabla \zeta=\sum_i \partial_{\tau_i} \zeta \tau_i
 	\]
and using \eqref{Dtaui}, we find
 	\begin{equation}\label{D2zeta}
 	D^2 \zeta(x)= \frac{1}{|x|^2}\lt( \nabla^2\zeta(\pi_{\S^2}(x)) - \pi_{\S^2}(x)\otimes \nabla \zeta(\pi_{\S^2}(x))+ \nabla \zeta(\pi_{\S^2}(x))\otimes\pi_{\S^2}(x)\rt).
 	\end{equation}
Let now $\rho$ be a standard convolution kernel with $\spt \rho\subset B_{1}$. For $\eps\in (0,1)$ we set $\rho_\eps(x)= \eps^{-3}\rho(x/\eps)$ and then
 	\[
 	 \psi_\eps(x)=\int_{\R^3} \rho_\eps(x-y) \psi(y) dy.
 	\]
 	Fix $x\in \S^2$. Using that
 	\[\int_{\R^3} \rho_\eps(y) dy=1\]
 	and for a.e. $y$ with $|y-x|\le \eps$, $D\psi(y)=|y|^{-1}\nabla\psi(\pi_{\S^2}(y))\in M$ provided   $\eps$ is small enough, we have
 	\begin{equation}\label{mainestimdistM}
 	\begin{aligned}
 	{\rm dist}_{\R^6}(\nabla \psi_\eps( x) ,M)^2&=\int_{\R^3}\rho_\eps(x-y){\rm dist}_{\R^6}(\nabla \psi_\eps( x),M)^2\,dy\\
 	&\le \int_{\R^3}\rho_\eps(x-y)|D\psi(y)-\nabla \psi_\eps( x)|^2\,dy\\
 	&\les  \int_{\R^3}\rho_\eps(x-y)\lt|D\psi(y)-\int_{\R^3} \rho_\eps(x-z) D\psi(z) dz\rt|^2\,dy\\
 	&\qquad\qquad  + \lt|\int_{\R^3} \rho_\eps(x-z) D\psi(z) dz-\nabla \psi_\eps( x)\rt|^2.
 		\end{aligned}
 	\end{equation}
 	We now bound separately the two right-hand side terms. For the first one, let $\zeta(y)=D\psi(y)$ so that with obvious notation
 	\begin{align*}
 	 \int_{\R^3}\rho_\eps(x-y)\lt|D\psi(y)-\int_{\R^3} \rho_\eps(x-z) D\psi(z) dz\rt|^2\,dy
 	&=
 	 \int_{\R^3}\rho_\eps(x-y)\lt|\zeta(y)-\zeta_\eps(x)\rt|^2\,dy\\
 	 &\les \eps^2 \int_{\R^3}\rho_\eps(x-y) |D\zeta(y)|^2 dy
 	\end{align*}
by the weighted Poincar\'e inequality. Since $D\zeta=D^2\psi$ we have by \eqref{D2zeta},
\begin{multline*}
 \int_{\R^3}\rho_\eps(x-y)\lt|\zeta(y)-\zeta_\eps(x)\rt|^2\,dy\les \eps^2 \int_{\R^3}\rho_\eps(x-y) \lt[|y|^{-2} (|\nabla^2\psi(\pi_{\S^2}(y))| +|\nabla \psi(\pi_{\S^2}(y))|)\rt]^2 dy\\
 \les \eps^{-1} \int_{B_\eps(x)}| \nabla^2\psi(\pi_{\S^2}(y))|^2 +|\nabla \psi(\pi_{\S^2}(y))|^2 dy.
\end{multline*}
Letting
\[
 G_\eps(x)=\int_{B_\eps(x)\cap \S^2} |\nabla^2 \psi|^2 +|\nabla \psi|^2d\H^2
\]
we thus have
\begin{equation}\label{firstrhs}
 \int_{\R^3}\rho_\eps(x-y)\lt|\zeta(y)-\zeta_\eps(x)\rt|^2\,dy\les G_\eps(x).
\end{equation}
Let us now turn to the second right-hand side term in \eqref{mainestimdistM}. Writing that
\[
 \psi_\eps(x)=\int_{\R^3} \rho_\eps(z) \psi(x-z) dz
\]
we compute by linearity of the projection,
\begin{equation}\label{nabpsieps}
 \nabla \psi_\eps(x)=\int_{\R^3} \rho_\eps(z) \pi_{T_x \S^2}\lt( D \psi(x-z)\rt) dz=\int_{\R^3} \rho_\eps(x-z) \pi_{T_x \S^2}\lt(D\psi(z)\rt) dz.
\end{equation}
Before proceeding further let us point out that for every $x,z\in \S^2$ and  $v\in T_y\S^2$, since $v\cdot y=0$ we have by writing $x= y+(x-y)$,
\begin{equation}\label{changeproj}
 |x\cdot v|=|(x-y)\cdot v|\le |x-y| |v|.
\end{equation}
We then have,
\begin{multline*}
 \lt|\int_{\R^3} \rho_\eps(x-z) D \psi(z) dz-\nabla \psi_\eps( x)\rt|\le \int_{\R^3} \rho_\eps(x-z)\lt|D \psi(z)-\pi_{T_x \S^2}\lt( D \psi( z)\rt)\rt|dz\\
 = \int_{\R^3} \rho_\eps(x-z)|z|^{-1}\lt|\nabla \psi(\pi_{\S^2}(z))-\pi_{T_x \S^2}\lt(  \nabla \psi( \pi_{\S^2}(z))\rt)\rt|dz
 =\int_{\R^3} \rho_\eps(x-z)|z|^{-1}\lt|\nabla \psi(\pi_{\S^2}(z))\cdot x\rt|dz
 \\
 \stackrel{\eqref{changeproj}}{\les} \eps \int_{\R^3} \rho_\eps(x-z)\lt|\nabla \psi(\pi_{\S^2}(z))\rt|dz.
\end{multline*}
 Since $\psi\in {\rm Lip}(\S^2)$ with $|\nabla \psi|=\sqrt{2}$, we conclude that
\begin{equation}\label{secondtermM}
 \lt|\int_{\R^3} \rho_\eps(x-z) D \psi(z) dz-\nabla \psi_\eps( x)\rt|\les \eps.
\end{equation}
Plugging \eqref{firstrhs} and \eqref{secondtermM} into \eqref{mainestimdistM} we find
\[
 {\rm dist}_{\R^6}(\nabla \psi_\eps( x) ,M)\les G_\eps(x)^{\frac{1}{2}}+ \eps.
\]
Now since  $|\nabla^2\psi|^2\in L^1(\S^2)$, we get that $ G_\eps(x)\to0$ as $\eps\to0$ uniformly in $ x$, and so
\[
\lim_{\eps\to 0} \sup_{ x\in\S^2}{\rm dist}_{\R^6}(\nabla \psi_\eps( x) ,M)=0.
\]
Thus for $\eps$ small enough, there exists for every $ x\in \S^2$ an element $A=A(x)=(a_1,a_2)\in M$  such that
\[
|\nabla \psi_\eps( x)-A|\le\frac{1}{16}.
\]
By bilinearity and the  triangular inequality we conclude that
\[
\begin{aligned}
	|\partial_1\psi_\eps\wedge \partial_2 \psi_\eps|&\ge |a_1\wedge a_2|-|(\partial_1 \psi_\eps-a_2)\wedge a_1|-|(\partial_2 \psi_\eps-a_1)\wedge a_1|-|(\partial_1 \psi_\eps-a_1)\wedge (\partial_2 \psi_\eps-a_2)|\\
	&\ge \frac14-|\partial_1 \psi_\eps-a_2|| a_1|-|\partial_2 \psi_\eps-a_1|| a_1|-|\partial_1 \psi_\eps-a_1|| \partial_2 \psi_\eps-a_2|\\
	&\ge\frac{1}{16}.
\end{aligned}
\]
This proves the first part of \eqref{quasiconfpsin}. For the second part of \eqref{quasiconfpsin}, we simply observe that by \eqref{nabpsieps},
\begin{multline*}
 |\nabla \psi_\eps(x)|\le \int_{\R^3} \rho_\eps(x-y) |\nabla \psi(\pi_{\S^2}(y))| dy +  \int_{\R^3} \rho_\eps(x-y) |x-y||D \psi(y)| dy\\
 \le \|\nabla \psi\|_{L^{\infty}(\S^2)}(1+C\eps).
\end{multline*}

Let us now prove that $\psi_\eps$ converges in $W^{2,2}(\S^2)$ to $\psi$ in $L^2(\S^2)$. By Poincar\'e inequality it is enough to prove that $\nabla^2\psi_\eps$ converges to $\nabla^2\psi$. Fix $x\in \S^2$ and write for simplicity $\tau_i=\tau_i(x)$. By \eqref{defnabla2} and the definition of $\psi_\eps$, we have
\begin{multline*}
 \nabla^2\psi_\eps(x)[\tau_i,\tau_j]=\int_{\R^3} \rho_\eps(x-y) \lt[D^2\psi(y)[\tau_i,\tau_j] -\delta_{ij} D\psi(y)\cdot x\rt]dy\\
 \stackrel{\eqref{D2zeta}}{=}\int_{\R^3} \rho_\eps(x-y)|y|^{-2}\nabla^2\psi(\pi_{\S^2}(y))[\tau_i,\tau_j]dy \\
 -\int_{\R^3} \rho_\eps(x-y)\lt[|y|^{-3}(\nabla \psi(\pi_{\S^2} (y))\cdot \tau_i) (y\cdot \tau_j) +|y|^{-3}( \nabla \psi(\pi_{\S^2} (y))\cdot \tau_j) (y\cdot \tau_i) +\delta_{ij} D\psi(y)\cdot x\rt]dy.
\end{multline*}
Using \eqref{changeproj} and recalling that $D\psi(x)\cdot x=0$ we have
\begin{multline*}
\lt| \nabla^2\psi_\eps(x)[\tau_i,\tau_j]-\nabla^2 \psi(x)[\tau_i,\tau_j]\rt|\les \lt|\int_{\R^3} \rho_\eps(x-y)\lt(|y|^{-2}\nabla^2\psi(\pi_{\S^2}(y))[\tau_i,\tau_j]- \nabla^2\psi(x)[\tau_i,\tau_j]\rt)dy\rt|\\
+ \eps \int_{\R^3} \rho_\eps(x-y) |\nabla \psi(\pi_{\S^2}(y))| dy\\
\les \lt|\int_{\R^3} \rho_\eps(x-y)\lt(\nabla^2\psi(\pi_{\S^2}(y))[\tau_i(\pi_{\S^2}(y)),\tau_j(\pi_{\S^2}(y))]- \nabla^2\psi(x)[\tau_i,\tau_j]\rt)dy\rt|\\
+ \eps \int_{\R^3} \rho_\eps(x-y) (|\nabla \psi(\pi_{\S^2}(y))| +|\nabla^2 \psi(\pi_{\S^2}(y))|) dy.
\end{multline*}
Thus
\begin{multline*}
 \lt| \nabla^2\psi_\eps(x)-\nabla^2 \psi(x)\rt|\les\lt|\int_{\R^3} \rho_\eps(x-y)(\nabla^2\psi(\pi_{\S^2}(y))- \nabla^2\psi(x))dy\rt|
\\+ \eps \int_{\R^3} \rho_\eps(x-y) (|\nabla \psi(\pi_{\S^2}(y))| +|\nabla^2 \psi(\pi_{\S^2}(y))|) dy.
\end{multline*}
Since $\nabla^2 \psi\in L^2(\S^2)$ implies as in the Euclidean case that
\[
 \lim_{\eps\to 0} \int_{\S^2} \lt|\int_{\R^3} \rho_\eps(x-y)(\nabla^2\psi(\pi_{\S^2}(y))- \nabla^2\psi(x))dy\rt|^2 d\H^2(x)=0,
\]
the $L^2$ convergence of $\nabla^2 \psi_\eps$ to $\nabla^2 \psi$ follows from
\begin{multline*}
 \int_{\S^2} \lt(\int_{\R^3} \rho_\eps(x-y)(|\nabla \psi(\pi_{\S^2}(y))| +|\nabla^2 \psi(\pi_{\S^2}(y))|) dy\rt)^2 d\H^2(x)\\
 \les \int_{\S^2} \lt(\eps^{-1} \int_{B_\eps(x)\cap \S^2}|\nabla \psi| +|\nabla^2 \psi| d\H^2\rt)^2d\H^2(x)\\
 \les \eps^2 \int_{\S^2} |\nabla \psi|^2 +|\nabla^2 \psi|^2 d\H^2.
\end{multline*}


 \end{proof}

\begin{lemma}\label{lem:approx}
		Assume that $E$ is such that $\partial E=\psi(\bSr)$ for some conformal embedding $\psi\in W^{2,2}(\bSr,\R^3)\cap {\rm Lip}(\bSr,\R^3)$  such that
 	\[
 	\|h-r\|_{L^\infty(\S^2)}\le\frac{r}{4}.
 	\]
 	Then, there exists a sequence of smooth maps $\psi_n:\bSr\to\R^3$ such that setting $\partial E_n=\psi_n(\bSr)$,
		\[
		\psi_n\longrightarrow \psi\ \text{in}\ W^{2,2}
		\qquad\text{and}\qquad  \lim_{n\to \infty}\int_{\partial E_n}|A|^2 d\H^2= \int_{\partial E} |A|^2 d\H^2.
		\]
		Moreover for all  $x\in \R^3$, $\rho>0$ such that $\mathcal H^2\left( \partial E\cap \partial B_\rho(x)\right)=0$ there holds
		\begin{equation}\label{locconv}
		\lim_{n\to \infty} \int_{\partial E_n\cap B_\rho(x)}|A|^2d\H^2=	\int_{\partial E\cap B_\rho(x)}|A|^2d\H^2.
		\end{equation}
		
	\end{lemma}
	\begin{proof}
	 By scaling we may assume that $r=1$. Let then $\psi_n$ be the sequence given by  Lemma \ref{lem:approx0}.
	Up to extraction, we may further assume that  $\nabla^2 \psi_n$ and $\nabla \psi_n$ converge a.e.. Recall also that since $\psi_n$ are smooth embeddings, by Jordan-Brouwer Theorem, $\psi_n(\bSr)=\partial E_n$ for some compact set $E_n$. Thus
	\[
\Phi_n:=	 |\partial_1 \psi_n\wedge \partial_2\psi_n|^{-3} \left|(\nabla^2 \psi_n \cdot N_{\psi_n})\begin{pmatrix} |\partial_2 \psi_n|^2 & -\partial_1\psi_n\cdot \partial_2\psi_n\\
-\partial_1\psi_n\cdot \partial_2\psi_n & |\partial_1 \psi_n |^2
\end{pmatrix}\right|^2
\]
converges a.e. to
\[
\Phi:=|\partial_1 \psi\wedge \partial_2\psi|^{-3} \left|(\nabla^2 \psi \cdot N_{\psi})\begin{pmatrix} |\partial_2 \psi|^2 & -\partial_1\psi\cdot \partial_2\psi\\
	-\partial_1\psi\cdot \partial_2\psi & |\partial_1 \psi |^2
\end{pmatrix}\right|^2.
\]
 By \eqref{quasiconfpsin} of Lemma \ref{lem:approx0} we get that
\[
 \Phi_n=|\partial_1 \psi_n\wedge \partial_2\psi_n|^{-3}\left|(\nabla^2 \psi_n \cdot N_{\psi_n})\begin{pmatrix} |\partial_2 \psi_n|^2 & -\partial_1\psi_n\cdot \partial_2\psi_n\\
-\partial_1\psi_n\cdot \partial_2\psi_n & |\partial_1 \psi_n|^2\end{pmatrix}\right|^2\les  |\nabla^2\psi_n|^2
\]
which allows us to conclude by dominate convergence that
\[
\int_{\partial E_n} |A|^2d\H^2=\int_{\bS^2}\Phi_n d\H^2\to \int_{\bS^2}\Phi d\H^2=\int_{\partial E} |A|^2d\H^2.
\]

Assume finally that $\mathcal H^2\left( \partial E\cap \partial B_\rho(x)\right)=0$. The measures $\mu_n= |A_n|^2 \H^2\restr \partial E_n$ weakly$^*$ converge to $\mu=|A|^2 \H^2\restr E_n$. Indeed, for every $f\in C^0_c(\R^3)$, we have by dominated convergence
\[
 \int_{\R^3} f d\mu_n= \int_{\S^2} f(\psi_n(x))\Phi_n(x) d\H^2(x)\to \int_{\S^2} f(\psi(x))\Phi(x) d\H^2(x)=\int_{\R^3} f d\mu.
\]
Since $\mu_n$ are Radon measures, and $\mu(\partial B_\rho(x))=0$ we have that $\mu_n(B_\rho(x))$ converges to $\mu(B_\rho(x))$, see e.g. \cite[Proposition 1.62]{AFP}.  This concludes the proof of \eqref{locconv}.
	\end{proof}

\begin{remark}\label{remrupp}
After this work was completed, we were informed by the authors of \cite{Rupp24} that a result analogous to the one in Lemma \ref{lem:approx} has been recently proved in \cite[Lemma 5.6]{Rupp24}, with a similar strategy. Indeed, in \cite{Rupp24} it is proved that
also (curvature) varifolds with Willmore energy below $8\pi$ can be smoothly approximated.
For the reader's convenience, we decided to keep our proof of this result.
\end{remark}

\subsection{The regularized capacitary functional}
Because of issues related to the continuity of $\Ia$ (see Remark \ref{remcontIa} below), it will be easier to work as in  \cite{gnr4,MurNov} with a regularized version. For $\eta>0$ and $\mu\in L^2(\R^3)$, we set
\[
 I_\alpha^\eta(\mu)=I_\alpha(\mu)+ \eta \int_{\R^3} \mu^2(x)\,dx=\int_{\R^3\times\R^3}\frac{
\mu(x)\,\mu(y)}{|x-y|^{3-\alpha}}\,dx\,dy + \eta\int_{\R^3}\mu^2(x)\,dx\,
\]
and then
\[
\Iae(E)=\min_{\mu\in L^2}\left\{I_\alpha^\eta(\mu)\,:\, 	 \int_{E} \mu(x)\,dx=1  \right\}.
\]
The following continuity result is implicitely contained in \cite{gnrI,gnr4}. For the reader's convenience we provide a proof.
\begin{lemma}\label{lem:contintuityIeta}
For $\eta>0$ let $E_n$ be a sequence of smooth compact sets converging in Hausdorff distance to a compact set $E$. Then $\I_\alpha^\eta(E_n)\to\I_\alpha^\eta(E)$. If $\eta=0$, the same conclusion holds provided  $E$ has interior density bounds i.e. there exists $r_0$ and $c>0$ such that for every $0<r\le r_0$ and every $x\in E$,
\[|E\cap B_r(x)|\ge c r^d.\]
\end{lemma}
	\begin{proof}
	 We start with the easier case $\eta>0$. We first prove lower-semicontinuity. Let $\mu_n$ be the  optimal charge distribution for $E_n$.  Notice that the existence and uniqueness of $\mu_n$ is a consequence of \cite[(1.4.5)]{landkof}. Since $E_n$ and $E$ are compact, up subsequence $\mu_n$ weakly converges as measure to some measure $\mu$. By semi-continuity of $I_\alpha$ under weak convergence, see \cite{landkof} as well as semi-continuity of the $L^2$ norm we have 
	 \[
	  \liminf_{n\to \infty} \Iae(E_n)=\liminf_{n\to \infty} I_\alpha^\eta(\mu_n)\ge I_\alpha^\eta(\mu).
	 \]
Since $\mu_n$ also converges weakly in $L^2$ to $\mu$ and $\chi_{E_n}$ converges strongly in $L^2$ to $\chi_E$ we get $\int_E \mu =1$ thus $\mu$ is admissible for $\Iae(E)$. This proves the lower-semicontinuity. For the upper semi-continuity, let $\mu$ be an optimal measure for $\Iae(E)$, let 
\[
c_n=\int_{E\cap E_n}\mu\to1,\quad\, \text{ as $n\to+\infty$},
\]
and then 
\[
\mu_n=c_n^{-1}\mu\chi_{E\cap E_n}.
\]
The measure $\mu_n$ is admissible for $\Iae(E_n)$ and we have 
\[
 \Iae(E_n)\le I_\alpha^\eta(\mu_n)\le c_n^{-2}I_\alpha^\eta(\mu)=c_n^{-2} \Iae(E).
\]
Sending $n\to \infty$ concludes the proof of the upper-semicontinuity.\\
If $\eta=0$, the semi-continuity argument works exactly as in the case $\eta>$ noticing as in \cite[Theorem 4.2]{gnrI} that by the Hausdorff convergence of $E_n$ to $E$, we have $\spt \mu\subset E$. For the upper-semicontinuity, we use\footnote{While \cite[Proposition 2.16]{gnrI} is stated for smooth sets, a quick inspection of the proof reveals that interior density bounds are sufficent.}   \cite[Proposition 2.16]{gnrI} (see also \cite[Theorem 3.11]{gnr4}) to infer that for every $\delta>0$, there exists $\mu_\delta\in L^\infty(E)$ such that
\[
 I_\alpha(\mu_\delta)\le \Ia(E)+\delta.
\]
We can now proceed as in the case $\eta>0$ by setting 
\[
 c_{n,\delta}=\int_{E\cap E_n}\mu_\delta,
\]
and then 
\[
\mu_{n,\delta}=c_{n,\delta}^{-1}\mu_\delta\chi_{E\cap E_n}.
\]
Since 
\[
 |1-c_{n,\delta}|\le \|\mu_\delta\|_{L^\infty} |E\backslash E_n| 
\]
we have $\lim_{n\to \infty} c_{n,\delta}= 1$. Arguing as above in combination with a diagonal argument concludes the proof.
\end{proof}

	\begin{remark}\label{remcontIa}
		Let  $\psi_n$ be a sequence of weak immersions  of $\S^2$ converging in  $W^{2,2}$ to an embedding $\psi$. It is tempting to conjecture but seemingly not so easy to prove, that $\Ia(\psi_n(\S^2))$ converges to $\Ia(\psi(\S^2))$. If we had this statement at hand we could avoid the introduction of $\Iae$ and directly work in the proof below with $\Ia$.
	\end{remark}

We close this section with a simple scaling property of $\Iae$.
\begin{lemma}\label{lem:scaling}
 For every $\alpha\in (0,3)$, $\eta\ge 0$, $E\subset \R^3$ compact and $\lambda>0$,
 \begin{equation}\label{eq:scalingIae}
  \Iae(\lambda E)\le \max(\lambda^{\alpha-3},\lambda^{-3}) \Iae(E).
 \end{equation}

\end{lemma}
\begin{proof}
 Let $\mu$ be the optimal measure for $\Iae(E)$ and let $\mu_\lambda=T_\lambda\sharp \mu$ where $T_\lambda(x)=\lambda x$. The measure $\mu_\lambda$ is admissible for $\Iae(\lambda E)$. We have $I_\alpha(\mu_\lambda)= \lambda^{\alpha-3} I_\alpha(\mu)$ and for $\eta>0$, since $\mu$ is absolutely continuous with respect to Lebesgue, $\mu_\lambda(x)=\lambda^{-3}\mu(x/\lambda)$ so that
 \[
  \int_{\R^3} \mu_\lambda^2=\lambda^{-3}\int_{\R^3} \mu^2.
 \]
We thus find
\[
 \Iae(\lambda E)\le I_\alpha(\mu_\lambda)+\eta \int_{\R^3} \mu_\lambda^2\le  \lambda^{\alpha-3} I_\alpha(\mu)+ \lambda^{-3}\eta\int_{\R^3} \mu^2\le \max(\lambda^{\alpha-3},\lambda^{-3}) \Iae(E).
\]
This concludes the proof of \eqref{eq:scalingIae}.
\end{proof}

	\subsection{Existence of minimizers}\label{sec:existence}
	For $\eta\ge 0$ we introduce the functionals
	\begin{equation}\label{defFeta}
	 \F_Q^\eta(E)= W(E)+ Q^2 \Iae(E)
	\end{equation}
so that $\F_Q=\F_Q^0$ and
\begin{equation}\label{deftildeFeta}
 \wF_Q^\eta(E)= \frac{1}{4} \int_{\partial E} |A|^2 d\H^2 + Q^2\Iae(E).
\end{equation}
When $\eta=0$ we simply write $\wF_Q=\wF_Q^0$.
Let us first prove that for $Q$ small enough, in order to prove that balls are the only volume-constrained minimizers of $\F_Q$, it is enough to prove that they are the only volume-constrained minimizers of $\wF_Q$.
\begin{lemma}\label{lem:A}
 There exists $Q_0>0$ such that for every $Q\le Q_0$ and every $\eta\in[0,1]$, every smooth compact set $E$ with $\F_Q^\eta(E)\le \F_Q^\eta(B_1)$ or $\wF_Q^\eta(E)\le \wF_Q^\eta(B_1)$ must have a  connected boundary with $\g(\partial E)=0$. In particular,
 \[
  \F_Q^\eta(E)=\wF_Q^\eta(E)+2\pi.
 \]
\end{lemma}
\begin{proof}
 Assume first that  $\F_Q^\eta(E)\le \F_Q^\eta(B_1)$ so that  if $Q$ is small enough
 \[W(E)\le W(B_1)+ Q^2\Iae(B_1)\le W(B_1)+ Q^2\I_\alpha^1(B_1) < 2\pi^2<8\pi.\]
  By Li-Yau inequality this implies that $\partial E$ is connected and  by \cite[Theorem A]{MaNe14} that $E$ is of sphere type i.e. $\g(\partial E)=0$. \\
 If instead, $\wF_Q^\eta(E)\le \F_Q^\eta(B_1)$ we may assume that $Q\ll 1$ is such that
		 \[
	 \wF_{Q}^\eta(B_1)=2 \pi +Q^2 \Iae(B_1)<4 \pi.
		\]
	Then,  \eqref{LiYauA} implies that $\partial E$ is connected with $\g(\partial E)=0$.
 The final claim is a consequence of \eqref{GaussBonnet}.
\end{proof}
%
	We   relax now the volume constraint. This type of relaxation is fairly common in isoperimetric type problem, see e.g. \cite{FuEsp,GolNo,gnr4}. For $\Lambda,\eta>0$ we set
	\[
	\wF^{\eta,\Lambda}_Q(E) =\wF^\eta_Q(E)+\Lambda | |E|-|B_1||
	\]
	and, for $\eta=0$, we set with a slight abuse of notation with respect to \eqref{deftildeFeta},
	\[
	 \wF_{Q}^\Lambda(E)=\wF^{0,\Lambda}_Q(E)=\wF_Q(E)+ \Lambda | |E|-|B_1||.
	\]
We begin by observing that minimizing $\wF^{\eta,\Lambda}_Q$ with  or without volume constraint is equivalent.
	\begin{proposition}\label{prop:Lambda}
	For every $Q_0>0$, there exists $0<\Lambda\les Q_0^2$ such that for every $Q\le Q_0$ and every $\eta\in [0,1]$,
		\begin{equation}\label{eq:relax}
		\inf \{\wF^{\eta,\Lambda}_Q(E) \ : \ |E|=|B_1|, \  E \ \textrm{smooth}\}=\inf \{\wF_Q^{\eta,\Lambda}(E) \ :  \  E \ \textrm{smooth}\}.
		\end{equation}
		Moreover, if $E$ is a minimizer for the right-hand side then $|E|=|B_1|$.
	\end{proposition}
	\begin{proof}
	Since in \eqref{eq:relax} the left-hand side is always larger than the right-hand side we only need to prove that provided $\Lambda=\Lambda_0 Q_0^2$ for some $\Lambda_0$ large enough, for every $Q\le Q_0$,
	\begin{equation}\label{eq:toproverelax}
	 \inf \{\wF^{\eta,\Lambda}_Q(E) \ : \ |E|=|B_1|, \  E \ \textrm{smooth}\}\le\inf \{\wF^{\eta,\Lambda}_Q(E) \ :  \  E \ \textrm{smooth}\}.
	\end{equation}
Let us assume that  there exists $E$ smooth with $|E|\neq |B_1|$ and for which
\begin{equation}\label{hyprelax}
 \wF^{\eta,\Lambda}_Q(E)\le \wF^{\eta,\Lambda}_Q(F) \qquad \textrm{for all smooth $F$ with } |F|=|B_1|.
\end{equation}
We first notice that since $\wF^{\eta,\Lambda}_Q(E)\le \wF^{\eta,\Lambda}_Q(B_1)$ and $\int_{\partial E} |A|^2 d\H^2\ge \int_{\partial B_1} |A|^2 d\H^2$ (recall \eqref{LiYauA}),
\[
 Q^2 \Iae(E)+\Lambda | |E|-|B_1||\le Q^2 \Iae(B_1)\le Q^2 \Ia^1(B_1).
\]
Recalling that $\Lambda= \Lambda_0 Q_0^2$, we find
\begin{equation}\label{smallrelax}
 \Lambda_0 | |E|-|B_1||\les 1 \qquad \textrm{and } \qquad \Iae(E)\le \Iae(B_1)\les 1.
\end{equation}
In particular, we may assume that $| |E|-|B_1||\ll1$ provided $\Lambda_0$ is chosen large enough. Let now $t= 1- (|B_1|/|E|)^{1/3}$ be such that $|(1-t)E|=|B_1|$. Since $\wF^{\eta,\Lambda}_Q(\lambda E)\le \wF^{\eta,\Lambda}_Q(E)$ for $\lambda\ge 1$  by \eqref{eq:scalingIae}, we have $t>0$. Moreover, since  $| |E|-|B_1||\ll1$,  we also have $t\ll1$. By assumption,
\[
 \wF^{\eta,\Lambda}_Q(E)\le\wF^{\eta,\Lambda}_Q( (1-t)E).
\]
Since
\begin{multline*}
 \wF^{\eta,\Lambda}_Q(E)=\frac{1}{4}\int_{\partial E} |A|^2d\H^2+Q^2 \Iae(E)+\Lambda ||E|-|B_1||\\
 = \frac{1}{4}\int_{\partial E} |A|^2d\H^2+Q^2 \Iae(E)+\Lambda |B_1| \lt( (1-t)^{-3}-1\rt)
\end{multline*}
and
\[
 \wF^{\eta,\Lambda}_Q( (1-t)E)=\frac{1}{4}\int_{\partial E} |A|^2d\H^2+Q^2\Iae( (1-t)E)\stackrel{\eqref{eq:scalingIae}}{\le} \frac{1}{4}\int_{\partial E} |A|^2d\H^2 +Q^2\Iae(E) (1-t)^{-3},
\]
we find after rearranging the terms
\[
 \Lambda |B_1| \lt( (1-t)^{-3}-1\rt)\le Q^2 \Iae(E)  \lt((1-t)^{-3}-1\rt).
\]
Since  $0<t<1$ we can simplify and  conclude (since $t\neq 0$)
\[
 \Lambda_0 Q_0^2=\Lambda \les Q_0^2 \Iae(E)\stackrel{\eqref{smallrelax}}{\les} Q_0^2.
\]
This proves that if $\Lambda_0$ is large enough then \eqref{hyprelax} cannot hold with $|E|\neq |B_1|$ and we indeed have \eqref{eq:toproverelax}.

	\end{proof}
	
	We now prove that for $Q$ small enough  $\wF^{\eta,\Lambda}_Q$ admits minimizers  and that they are $W^{2,2}$ embeddings of the sphere.
	
	\begin{proposition}\label{prop:exist}
		There exists $Q_0>0$ such that for every $Q\le Q_0$ and every $0<\eta\le 1$, there exist $r>0$ and $\psi\in W^{2,2}(\bSr)\cap {\rm Lip}(\bSr)$ a conformal embedding such that (the implicit constants do not depend on $\eta$)
\begin{equation}\label{quantpsi}
  |r-1|+\|\psi-{\rm Id}\|^2_{W^{2,2}(\bSr)}+\|h-1\|_{L^\infty(\bSr)}\les Q^2
\end{equation}
and such that $\partial E=\psi(\bSr)$ satisfies  $|E|=|B_1|$ and
\begin{equation}\label{equalinf}
 \wF_Q^{\eta,\Lambda}(E)= \inf\{\wF_Q^{\eta,\Lambda}(F) \ : \   F \ \textrm{smooth}\}.
\end{equation}
Moreover, there exists $E_n$ smooth with $\partial E_n$ converging to $\partial E$ in the Hausdorff sense and  such that
\begin{multline}\label{convsmooth}
 \wF_Q^{\eta,\Lambda}(E_n)\to \wF_Q^{\eta,\Lambda}(E) \qquad \textrm{with } \\ \int_{\partial E_n\cap B_\rho(x)} |A|^2 d\H^2\to \int_{\partial E\cap B_\rho(x)} |A|^2 d\H^2 \quad \textrm{ provided } \quad  \H^2(\partial E\cap \partial B_\rho(x))=0.
\end{multline}
Finally, if $E$ is a (smooth) minimizer of \eqref{equalinf} for some $\eta\in[0,1]$, then up to translation, $\partial E=\psi(\bSr)$ for some $r$ and $\psi$ satisfying \eqref{quantpsi}.
	\end{proposition}
	\begin{proof}
	Let $Q_0$ be given by Lemma \ref{lem:A}. If $E$ is smooth  set with $|E|=|B_1|$ and  such that $\wF_Q^{\eta}(E)\le \wF_Q^{\eta}(B_1)$, since $\partial E$ is connected with $\g(\partial E)=0$, using \eqref{GaussBonnet} we find
	\[
\frac{1}{4}\int_{\partial E} |A|^2 d\H^2-\frac{1}{4}\int_{\partial B_1} |A|^2 d\H^2=W(E)-W(B_1)=\frac{1}{2}\int_{\partial E} |A^\circ|^2 d\H^2.
	\]
Moreover, by \cite[Proposition 4.3]{GolNovRog}, $\H^2(\partial E)$ is universally bounded so that by \cite{roesc12}, if $r$ is such that $\H^2(\bSr)=\H^2(\partial E)$,
\[
 |r-1|\les (\H^2(\partial E)-4\pi)\les \int_{\partial E} |A^\circ|^2\les Q^2 \I_\alpha^\eta(B_1)\le Q^2 \I_\alpha^1(B_1)\lesssim Q^2.
\]
Notice that by the isoperimetric inequality, $\H^2(\partial E)\ge \H^2(\partial B_1)$ and thus $r\ge 1$.
Appealing to \cite{DeLMu2005,DeLMu2006}, and up to translation, we find the existence of a conformal embedding $\psi\in W^{2,2}(\bSr)$ satisfying 
\[
 \|\psi-{\rm Id}\|^2_{W^{2,2}(\bSr)}+\|h-r\|_{L^\infty(\bSr)}\les \int_{\partial E} |A^\circ|^2\les Q^2 .
\]
By triangle inequality, this proves the last part of the statement.\\
If now $E_n$ is a minimizing sequence then by the previous discussion, we find after possible translation, a sequence $r_n$ uniformly bounded from above and below and a sequence of conformal $\psi_n\in W^{2,2}(B_{r_n})$ satisfying \eqref{quantpsi} and such that $ \partial E_n=\psi_n(\partial B_{r_n})$. 
After extraction we find that $r_n\to r>0$, $\psi_n\rightharpoonup \psi$ weakly in $W^{2,2}$ with $(r,\psi)$ satisfying \eqref{quantpsi}. By Sobolev embedding, $\psi_n$ also converges strongly  in $W^{1,2}$ and in $C^{0,\beta}$ for every $\beta\in (0,1)$ so that $\psi$ is conformal. By the strong $W^{1,2}$ convergence and \eqref{quantpsi} we have
\[
 \|h-1\|_{L^\infty(\bSr)}\les  Q^2 .
\]
In particular, $\psi \in {\rm Lip}(\bSr)$ and up to decreasing the value of $Q_0$ we may assume that $h\in (1/2,2)$.
By the $C^{0,\beta}$ convergence of $\psi_n$, we have that $\partial E_n$ converges in the Hausdorff topology to $\partial E=\psi(\bSr)$.  By lower semicontinuity of the Willmore energy we have
\[
 W(E)\le \liminf_{n\to \infty} W(E_n)< 8\pi
\]
and thus by Li-Yau inequality \cite[(A.17)]{kusc04} $\psi$ is an embedding.
By the Hausdorff convergence of  $\partial E_n$ to $ \partial E$, we find  $|E|=|B_1|$ and $\lim_n \Iae(E_n)=\Iae (E)$, see Lemma \ref{lem:contintuityIeta}, so that
\begin{equation}\label{lowerequalinf}
 \wF_{Q}^{\eta,\Lambda}(E)\le \liminf_{n} \wF^{\eta,\Lambda}_Q(E_n)= \inf\{\wF^{\eta,\Lambda}_Q(F) \ :   F \ \textrm{smooth}\}.
\end{equation}
To conclude the proof we use that since $h\in (1/2,2)$, we can apply Lemma \ref{lem:approx} and find a new sequence $\widetilde{\psi}_n$ of smooth embeddings of $\bSr$ converging strongly in $W^{2,2}$ to $\psi$. Setting $\partial \widetilde{E}_n=\widetilde{\psi}_n(\bSr)$  we find that  $\partial \widetilde{E}_n$ converges in the Hausdorff topology to $\partial E$ with $\int_{\partial E_n} |A|^2 d\H^2\to \int_{\partial E} |A|^2d\H^2$ and $|\widetilde{E}_n|\to |E|$. Moreover, the second part of \eqref{convsmooth} holds. Finally, by Lemma \ref{lem:contintuityIeta} we have $\Iae(\widetilde{E}_n)\to \Iae(E)$ which yields $\wF^{\eta,\Lambda}_Q(\widetilde{E}_n)\to \wF^{\eta,\Lambda}_Q(E)$. This proves the first part of \eqref{convsmooth} as well as
\[
 \wF^{\eta,\Lambda}_Q(E)=\lim_{n\to \infty} \wF^{\eta,\Lambda}_Q(\widetilde{E}_n)\ge \inf \{\wF^{\eta,\Lambda}_Q(E) \ :  \  E \ \textrm{smooth}\}.
\]
Combined with \eqref{lowerequalinf} and \eqref{eq:relax}, this concludes the proof of \eqref{equalinf}.

\end{proof}

	\section{Regularity}\label{sec:regularity}

	We prove now an elementary but useful decay estimate (see e.g. \cite[Lemma 5.13]{giamar} for a variant of this result).
	\begin{lemma}\label{lem:induction}
	 Let $\Lambda, \delta>0$, $r_0\le 1$, $\theta\in (0,1)$ and $\gamma\in (0,1)$. Let $\psi$ be a positive and increasing function such that
	 \begin{equation}\label{hyp:inductionpsi}
	  \psi(\theta r)\le \gamma \psi(r) +\Lambda r^\delta \qquad \forall r\in (0,r_0).
	 \end{equation}
Then, there exists $\beta=\beta(\theta,\gamma,\delta)\in (0,\delta)$ and $C=C(\theta,\gamma,\delta)>0$ such that
\begin{equation}\label{eq:induction}
 \psi(r)\le C\lt(\frac{r}{r_0}\rt)^\beta (\psi(r_0)+\Lambda r_0^\beta).
\end{equation}

	\end{lemma}
\begin{proof}
We fix $\beta\in (0,\delta)$ such that
\begin{equation}\label{eq:beta}
 \gamma< \theta^\beta.
\end{equation}
Notice that since $\gamma<1$ this is always possible by taking $\beta$ small enough. We now prove that there exists $C= C(\beta,\theta)$ such that \eqref{eq:induction} holds. For $k\ge 0$ let $r_k=\theta^k r_0$. Since $\psi$ is increasing, for every $r\in (r_{k+1},r_k)$, $\psi(r_{k+1})\le \psi(r)\le \psi(r_k)$ so that it is enough to prove \eqref{eq:induction} for $r=r_k$. We now argue by induction. For $k=0$ the statement clearly holds (with $C=1$). Assume that it holds for $r_k$. Then, using \eqref{hyp:inductionpsi} we have
\[
		\begin{aligned}
			\psi( r_{k+1})&\le \gamma \psi(r_k)+\Lambda r_k^\delta
			\\
			&\le
			\gamma C\left(\frac{r_k}{r_0}\right)^\beta\left(\psi(r_0)+\Lambda r_0^\beta\right)+\Lambda r_k^\delta\\
			&=C_\beta\left(\frac{r_{k+1}}{r_0}\right)^\beta\left(\psi(r_0)+\Lambda r_0^\beta\right)\left(\frac{\gamma}{\theta^\beta}+\frac{1}{C }\frac{r_k^\delta}{r_{k+1}^\beta}\right)\\
			&=C\left(\frac{r_{k+1}}{r_0}\right)^\beta\left(\psi(r_0)+\Lambda r_0^\beta\right)\left(\frac{\gamma}{\theta^\beta}+\frac{1}{C\theta^\beta }r_0^{\delta-\beta}\right)\\
			&\stackrel{r_0\le 1}{\le} C\left(\frac{r_{k+1}}{r_0}\right)^\beta\left(\psi(r_0)+\Lambda r_0^\beta\right)\left(\frac{\gamma}{\theta^\beta}+\frac{1}{C\theta^\beta }\right).
		\end{aligned}
		\]
		Since $\theta$ and $\beta$ are fixed, by \eqref{eq:beta} we can find $C$ large enough such that
		\[
		 \frac{\gamma}{\theta^\beta}+\frac{1}{C\theta^\beta }\le 1.
		\]
This concludes the proof of \eqref{eq:induction} for $r_{k+1}$.
\end{proof}

	We are now in position to show the main regularity theorem we shall need in the sequel. 

	\begin{definition}\label{defminapprox}
	 For $\eta\ge 0$, we say that $E$ is a (volume constrained) approximable minimizer of $\wF^{\eta}_Q$ if $|E|=|B_1|$,  and up to translation there exist, $r>0$, $\psi\in W^{2,2} (\bSr)\cap {\rm Lip}(\bSr)$ and $E_n$ smooth such that $\partial E=\psi(\bSr)$, $\partial E_n$ converges to $\partial E$ in Hausdorff distance, \eqref{convsmooth} holds and
	 \[
	  \wF^{\eta}_Q(E)= \inf\{\wF^{\eta,\Lambda}_Q(F) \ :  \  F \ \textrm{smooth}\}.
	 \]

	\end{definition}

	\begin{remark}
	 As a consequence of Proposition \ref{prop:exist}, for every $\eta>0$ there exists approximable minimizers  of $\F^{\eta,\Lambda}_Q$.
	\end{remark}

We write every $x\in \R^3$ as $x=(x',x_3)$ and denote $B'_r=\{|x'|< r\}$ the two-dimensional ball of radius $r$.
	\begin{theorem}\label{thm:regularity}
		Let $\eta\in[0,1]$ and let $Q_0>0$ be the constant provided by Proposition \ref{prop:exist}. Then there exists $\eps>0$,  and $\beta\in (0,1)$ not depending on $\eta$ such that for every $Q\le Q_0$, if $E$ is an approximable minimizer  of $\wF_Q^{\eta,\Lambda}$  such that $0\in \partial E$ and for some  $r_0\in (0,1)$,
		\begin{equation}\label{hypepsreg}
		\int_{\partial E\cap B_{r_0}}|A|^2 d\H^2\le\varepsilon^2,
		\end{equation}
		then $\partial E\cap B_{ r_0/2}$ is a $C^{1,\beta}$ surface in the sense that up to a rotation,
		\[
		\partial E\cap B_{ r_0/2}=\{(x',u(x'))\,:\, u\in C^{1,\beta}(B'_{ r_0/2})\}.
		\]
		Moreover there holds 
		\begin{equation}\label{C1betaestim}
		r_0^\beta \|u\|_{C^{1,\beta}(B_{r_0/2})}\lesssim 1.
  \end{equation}
	As a consequence every approximable minimizer has boundary of class $C^{1,\beta}$.
	\end{theorem}
	\begin{proof}
		The proof follows closely the arguments from \cite{Simon93}.
	Let $E_n$ be a sequence of smooth sets such that $\partial E_n$ converges to $\partial E$ in Hausdorff distance and \eqref{convsmooth} holds. In particular, for every $1\ge \rho>0$ and every smooth set $F_n$ such that $F_n\Delta E_n\subset B_\rho$,  setting
	\[\delta_n=\wF_Q^{\eta,\Lambda}(E_n)-\wF_Q^{\eta}(E)\to 0,
	\]
	we have
	\[
	 \wF_Q^{\eta,\Lambda}(E_n)= \wF_Q^{\eta}(E)+\delta_n\le \wF_Q^{\eta,\Lambda}(F_n)+ \delta_n.
	\]
Rearranging terms we get
\begin{equation}\label{quasimin}
 \int_{\partial E_n \cap B_\rho} |A|^2d\H^2\le \int_{\partial F_n \cap B_\rho} |A|^2d\H^2 + Q^2(\Iae(E_n)-\Iae(F_n))+\Lambda |E_n\Delta F_n| +\delta_n.
\end{equation}
We now claim that
\begin{equation}\label{claim:difIalpha}
 \I_\alpha^\eta(E_n)-\I_\alpha^\eta(F_n)\les \rho^{3-\alpha},
\end{equation}
where the implicit multiplicative constant is independent of $\eta$.
For this, we appeal to \cite[Lemma 3.5]{gnr4} (see also \cite[Lemma 2]{MurNovRuf} or \cite[Lemma 13]{MurNovRuf2}) to get that
		\[
		\begin{aligned}
		 \I_\alpha^\eta(E_n)-\I_\alpha^\eta(F_n)&\le \I_\alpha^\eta(E_n\cap F_n)-\I_\alpha^\eta(F_n)\\
		&\le \frac{(\I_\alpha^\eta(E_n\cap F_n))^2}{\I_\alpha^\eta(F_n\setminus E_n)}.
\end{aligned}
		\]
		On the one hand, by Hausdorff convergence there exists $\bar x\in \R^3$ such that $B_{1/2}(\bar x)\subset E_n\cap F_n$ for $n$ large enough and thus $\I_\alpha^\eta(E_n\cap F_n)\le \I_\alpha^\eta(B_{1/2})\les 1$ (recall that $\eta\le1$). On the other hand, since $F_n\setminus E_n \subset B_{\rho}$, we have
		\[
		\frac{1}{\I_\alpha^\eta(F_n\setminus E_n)}\le \frac{1}{\I_\alpha^\eta(B_\rho)}\le \frac{1}{\I_\alpha(B_\rho)}\les \rho^{3-\alpha}.
\]
This concludes the proof of \eqref{claim:difIalpha}. Combining this with \eqref{quasimin}, $|E_n\Delta F_n|\les \rho^3$  and the fact that $\Lambda \les Q_0^2$, we obtain the quasi-minimality property,
\begin{equation}\label{quasimin2}
 \int_{\partial E_n \cap B_\rho} |A|^2d\H^2\le \int_{\partial F_n \cap B_\rho} |A|^2d\H^2 + CQ_0^2 \rho^{3-\alpha}  +\delta_n,
\end{equation}
for every $\rho\in (0,1]$ and every smooth set $F_n$ with $E_n \Delta F_n\subset B_\rho$.\\

We now recall that sending $\rho\to \infty$  in \cite[(1.3)]{Simon93} we have for every $\rho>0$,
\begin{equation}\label{upperdensityEn}
 \H^2(\partial E_n\cap B_\rho)\les \rho^2.
\end{equation}
Therefore, there exists  a universal $\eps>0$ such that \cite[Lemma 2.1]{Simon93} applies to $E_n$ with $2\eps$. Assume that \eqref{hypepsreg} holds for $r_0>0$ and let $r\le r_0$ (so that \eqref{hypepsreg} also holds for $r$). Since $\H^2(\partial E\cap \partial B_\rho)=0$ for a.e. $\rho$, using the second part of \eqref{convsmooth}, we see that  \eqref{hypepsreg}  also holds for $E_n$ with $2\eps$ provided $n$ is large enough.  Moreover, there exists $\rho\in (r/2,r)$ for which
\[
 \H^1(\partial E_n\cap \partial B_\rho)\le \frac{2}{r} \H^2(\partial E_n\cap B_r)\les \rho. 
\]
In particular there is a universal $\theta\in(0,1/2)$ such that by \cite[Lemma 1.4]{Simon93}, $\partial E_n\cap B_{\theta \rho}$ is connected. By \cite[Lemma 2.1]{Simon93}, we can find $\sigma \in (\theta \rho/2, \theta \rho)$ such that $\partial E_n\cap B_\sigma$ is a topological disk $D_1^{(n)}$ with $\partial E_n\cap \partial B_\sigma$ coinciding with the graph of a function $u_n$. Considering the surface $\widetilde{\Sigma}_n=(\partial E_n\backslash D_1^{(n)})\cup {\rm graph }\, w_n$ where $w_n$ is the biharmonic function coinciding with $u_n$  on $ \partial B_\sigma$ we see that $\widetilde{\Sigma}_n=\partial F_n$ for some $C^{1,1}$ set $F_n$ with $F_n\Delta E_n\subset B_{\sigma}$. Moreover, by \cite[Lemma 2.2]{Simon93}, we have arguing exactly as in \cite{Simon93},
\[
  \int_{\partial F_n\cap B_{\sigma}} |A|^2d\H^2\le C \int_{\partial E_n \cap (B_{\theta \rho}\backslash B_{\theta\rho/2})} |A|^2d\H^2.
\]
Combining this with \eqref{quasimin2} we find 
\[
 \int_{\partial E_n \cap B_\sigma} |A|^2d\H^2\le C \int_{\partial E_n \cap (B_{\theta \rho}\backslash B_{\theta\rho/2})} |A|^2d\H^2 + CQ_0^2 \rho^{3-\alpha}  +\delta_n.
\]
Adding $C$ times the left-hand side to this inequality and dividing by $1+C$ we get 
\[
 \int_{\partial E_n \cap B_{\theta \rho/2}} |A|^2d\H^2 \le \int_{\partial E_n \cap B_\sigma} |A|^2d\H^2 \le \gamma \int_{\partial E_n \cap B_{\theta \rho}} |A|^2 d\H^2+ CQ_0^2 \rho^{3-\alpha}  +\delta_n,
\]
where $\gamma=C/(C+1)\in (0,1)$. Setting 
\[\psi(\rho)=\int_{\partial E\cap B_\rho} |A|^2d\H^2\]
and sending $n\to \infty$  in the previous inequality, thanks to \eqref{convsmooth} we find (recall that $\rho\in (r/2,r)$ and that $\H^2(\partial E\cap \partial B_\rho)=0$ for a.e. $\rho$)
\[
 \psi(\theta r/4)\le \psi(\theta \rho/2)\le \gamma\psi(\theta \rho) + CQ_0^2 \rho^{3-\alpha}\le \gamma \psi(r) +CQ_0^2 r^{3-\alpha}.
\]
Applying Lemma \ref{lem:induction} with $\theta'=\theta/4$, we get the existence of $\beta\in (0,3-\alpha)$ such that
\[
 \psi(r)\les \lt(\frac{r}{r_0}\rt)^\beta (\psi(r_0)+ r_0^\beta)\les \lt(\frac{r}{r_0}\rt)^\beta.
\]
Arguing as for \cite[(3.5)]{Simon93}, we find that up to a rotation,
\[
 \partial E\cap B_{r_0/2}=\{(x',u(x')), x'\in B'_{r_0/2}\}
\]
and for every $r\le r_0/2$, there exists $\tau=\tau(r)$ such that 
\[
 \frac{1}{r^2}\int_{B'_r} |Du-\tau|^2 \les \lt(\frac{r}{r_0}\rt)^{\beta/2}.
\]
Repeating the argument for every $x\in \partial E\cap B_{r_0/2}$ we find by Campanato criterion  that $Du\in C^{0,\beta/4}(B'_{r_0/2})$ with 
\[
 r_0^{\frac{\beta}{4}} \|u\|_{C^{1,\beta/4}(B_{r_0/2})}\les 1.
\]
Up to renaming $\beta$ this concludes the proof of \eqref{C1betaestim}.

\end{proof}

	\begin{corollary}\label{cor:ns}
		There exists $Q_0>0$ such that for $Q<Q_0$ and $\eta\in[0,1]$ any approximable minimizer of $\wF_Q^{\eta,\Lambda}$ is a $C^{1,\beta}$ regular nearly spherical set with uniformly bounded (with respect to $Q$ and $\eta$) $C^{1,\beta}$ norm.
	\end{corollary}
	\begin{proof}
		Let us notice that for any $\varepsilon>0$ one can  find $r_0>0$ such that for every $x\in \S^2$,
		\begin{equation}\label{smallAball}
		\int_{ \S^2\cap B_{r_0}(x)} |A|^2d\H^2\le \varepsilon^2/2.
		\end{equation}
		We claim that if $Q_0$ is small enough then for every $Q\le Q_0$ and every approximable $E$ of $\wF_Q$, up to translation 
		\[
		 \partial E=\{(1+\phi(x))x \ : x\in \S^2\}
		\]
for some function $\phi$ with $\|\phi\|_{C^{1,\beta}(\S^2)}\les 1$, where $\beta>0$ is given by Theorem \ref{thm:regularity}. Indeed, assume this is not the case then by a covering argument and Theorem \ref{thm:regularity} there would exist $\bar x\in \S^2$ and  a sequence $E_Q$ of approximable minimizers of $\wF_Q^{\eta_Q}$ with $Q\to 0$ and such that $\bar x\in \partial E_Q$ with
\begin{equation}\label{hypabsurdC1beta}
 \int_{ \partial E_Q \cap B_{r_0}(\bar x)} |A|^2 d\H^2\ge \eps^2.
\end{equation}

However, by Proposition \ref{prop:exist}, for every $Q$ there exist  $r_Q, \psi_Q$ such that  $\psi_Q$ is conformal, $\partial E_Q=\psi_Q(\partial B_{r_Q})$ and 
\[
 |r_Q-1|+\|\psi_Q-{\rm Id}\|^2_{W^{2,2}(\partial B_{r_Q})}+\|h_Q-1\|_{L^\infty(\partial B_{r_Q})}\les Q^2.
\]
Thus $\psi_Q$ converges $W^{2,2}$ to $ {\rm Id}$ and $\int_{\partial E_Q} |A|^2 d\H^2\to \int_{\S^2} |A|^2 d\H^2$ which then implies that also
\[
 \int_{ \partial E_Q \cap B_{r_0}(\bar x)} |A|^2 d\H^2\to \int_{ \S^2 \cap B_{r_0}(\bar x)} |A|^2 d\H^2.
\]
In light of \eqref{smallAball} and \eqref{hypabsurdC1beta} this gives a contradiction if $Q$ is small enough.
	\end{proof}
	                                                                                                                                                                                                                                                                                                                                                                                                                                                                    
	\section{Proof of the main result}\label{sec:mainthm}
	We now prove the main result of the paper.
	
		\begin{proof}[Proof of Theorem \ref{thm:main}]
	Let $Q_0>0$ be the constant given in Proposition \ref{prop:exist}. Let $E^\eta$ be an approximable minimizer for $\wF_Q^{\eta}$, for $\eta>0$. Up to decrease further $Q$ we can suppose that the hypotheses of Theorem \ref{thm:regularity} applies so that by  Corollary \ref{cor:ns}  we know that \[
	\partial E^{\eta}=\{(1+\phi_\eta(x))x \ : x\in \S^2\}
	\]
	for some $\phi_\eta$ with $\|\phi_\eta\|_{C^{1,\beta}(\S^2)}$ uniformly bounded in $\eta$ and $Q$.

	By Arzela-Ascoli we can find, for any $\widetilde\beta<\beta$,  a sequence $\eta_n>0$ with $\eta_n\to0$ such that $\phi_{\eta_n}$ converges to $\phi$ in $C^{1,\widetilde \beta}(\S^2)$. As a consequence, setting $E_n=E^{\eta_n}$ and
	\[
	\partial E=\{(1+\phi(x))x \ : x\in \S^2\}
	\]
	we have that $E_n\to E$ in Hausdorff and in $L^1$. Notice that since $|E^\eta|=|B_1|$ we also have $|E|=|B_1|$.

		Let us prove that $E$ is an approximable minimizer of $\wF_Q$. Since $E_n$ are approximable minimizers, there exists $r_n$ and $\psi_n\in W^{2,2}(\S^2_{r_n})\cap {\rm Lip}(\S^2_{r_n})$ with $\partial  E_n=\psi_n(\S^2_{r_n})$,  such that  \eqref{quantpsi}  holds.
		Let $\psi$ be the limit of $\psi_n$ and $r=\lim_{n} r_n$. By Proposition \ref{lem:approx}, Lemma \ref{lem:contintuityIeta} and the fact that $E$ is $C^{1,\beta}$, there exists a sequence of smooth sets $\widetilde{E}_n$ such that \eqref{convsmooth} holds with $\eta=0$. In particular we have
		\[
		 \widetilde{\F}_Q(E)\ge \inf\lt\{ \widetilde{\F}^\Lambda_Q(F) \ : \ F \textrm{ smooth}\rt\}.
		\]
		Let us prove the converse inequality. Since $\Iae(E_n)\ge \Ia(E_n)$ and $\Ia(E_n)\to \Ia(E)$ by Lemma \ref{lem:contintuityIeta} again and the fact that $E$ is $C^{1,\beta}$.   By lower semicontinuity of $\int_{\partial E} |A|^2 d\H^2$ we find
		\begin{equation}\label{lowerEn}
		 \widetilde{\F}_Q(E)\le \liminf_{n\to \infty} \wF_Q^{\eta_n,\Lambda}(E_n).
		\end{equation}
 If  now $F$ is a smooth compact set then using the fact that $E_n$ are approximable minimizers and \eqref{lowerEn} we have
\[
  \widetilde{\F}_Q(E)\le \liminf_{n\to \infty} \wF_Q^{\eta_n}(E_n)\le\liminf_{n\to \infty} \wF_Q^{\eta_n,\Lambda}(F).
\]
Using \cite[Proposition 2.16]{gnrI} and a diagonal argument we see that for smooth sets $F$ we have $\lim_{n\to \infty} \wF_Q^{\eta_n,\Lambda}(F)= \wF_Q^{\Lambda}(F)$. We conclude that
\[
 \widetilde{\F}_Q(E)= \inf\lt\{ \widetilde{\F}^\Lambda_Q(F) \ : \ F \textrm{ smooth}\rt\}.
\]
By Lemma \ref{lem:A} and Proposition \ref{prop:Lambda} we conclude that we also have
\[
  {\F}_Q(E)= \inf\lt\{ \F_Q(F) \ : \ |F|=|B_1|,  \ F \textrm{ smooth}\rt\}.
\]
This proves the existence of approximable (volume-constrained) minimizers of $ {\F}_Q$. Moreover, Corollary \ref{cor:ns} implies that all such approximable minimizers are $C^{1,\beta}$ regular nearly-spherical sets with $C^{1,\beta}$ norm uniformly converging to $0$ as $Q\to 0$. Let now $E$ be any such approximable minimizer.
Testing the minimality against $F=B_1$ we have after rearranging the terms
\begin{equation}\label{delWdelI}
 W(E)-W(B_1)\le Q^2(\Ia(B_1)-\Ia(E)).
\end{equation}
Let us bound the left-hand side from below. Since $E$ is an approximable minimizer there exists a sequence  of smooth sets $E_n$ with $W(E_n)\to W(E)$ and $P(E_n)\to P(E)$. Let $r_n\to 1$ be such that $|B_{r_n}|=|E_n|$. By \cite{roesc12},
	\[
	P(E_n)-P(B_{r_n})\les (W(E_n)-W(B_{r_n})),
	\]
	so that sending $n\to \infty$ we find
	\begin{equation}\label{delW}
	 P(E)-P(B_1)\les (W(E)-W(B_1)).
	\end{equation}
Using that $E$ is a nearly spherical set, we have by \cite[Proposition 5.5]{gnrI} in the case $\alpha=2$ (see also \cite[Lemma 2.2]{Prunier} for a stronger statement) and by \cite[Proposition 4.5.]{gnr4} for  $\alpha\in (0,2)$ that the right-hand side of \eqref{delWdelI} may be bounded from above by
\begin{equation}\label{delI}
 \Ia(B_1)-\Ia(E)\les ( P(E)-P(B_1)).
\end{equation}
Plugging  \eqref{delW} and \eqref{delI} into \eqref{delWdelI} yields
\[
  (P(E)-P(B_1))\les Q^2  (P(E)-P(B_1))
\]
which for $Q$ small enough implies $P(E)=P(B_1)$ and thus up to translation, $E=B_1$.
	\end{proof} 
%
%

	
	\bibliographystyle{acm}

	\bibliography{bibliognrW}	
	
\end{document}